\documentclass[letterpaper, 10 pt, conference]{ieeeconf}  

\IEEEoverridecommandlockouts                              
\usepackage{blindtext}
\usepackage{graphicx}
\usepackage{amsmath}
\usepackage{amssymb}
\usepackage{xcolor}
\usepackage{caption}
\usepackage{subcaption}

\newtheorem{lemma}{Lemma}

\newtheorem{assumption}{Assumption}
\newtheorem{proposition}{Proposition}

\newcommand{\E}{\mathbb{E}}

\usepackage{bbm}
\usepackage{dsfont}
\usepackage{bm}

\newcommand{\matilde}[1]{{\color{magenta} Maty: #1}}

\title{\LARGE \bf
On the Synthesis of Bellman Inequalities for \\ Data-Driven Optimal Control
}

\author{Andrea Martinelli, Matilde Gargiani and John Lygeros
\thanks{Research supported by the European Research Council
	under the H2020 Advanced Grant no. 787845 (OCAL).}
\thanks{The authors are with the Automatic Control Laboratory, Swiss Federal Institute of Technology (ETH) Zurich, 8092 Zurich, Switzerland. Emails: {\tt\small \{andremar,gmatilde,lygeros\}@ethz.ch}}%
}

\begin{document}

\maketitle
\thispagestyle{empty}
\pagestyle{empty}


\begin{abstract}

In the context of the linear programming (LP) approach to data-driven control, one assumes that the dynamical system is unknown but can be observed indirectly through data on its evolution. Both theoretical and empirical evidence suggest that a desired suboptimality gap is often only achieved with massive exploration of the state-space. In case of linear systems, we discuss how a relatively small but sufficiently rich dataset can be exploited to generate new constraints offline and without observing the corresponding transitions. Moreover, we show how to reconstruct the associated unknown stage-costs and, when the system is stochastic, we offer insights on the related problem of estimating the expected value in the Bellman operator without re-initializing the dynamics in the same state-input pairs. 

\end{abstract}


\section{INTRODUCTION}


The linear programming (LP) approach to optimal control problems was initially developed by A.S. Manne in the 1960s \cite{ManneLP}, following the well-known studies conducted by R. Bellman in the 1950s \cite{BellmanDP}. The idea is to exploit the monotonicity and contractivity properties of the Bellman operator \cite{BertsekasVol2} to build LPs whose solution is the optimal value function. An evident advantage of the LP formulation is that there exist efficient and fast algorithms to tackle such programs \cite{BoydConvexOptimization}. On the other hand, similarly to the classic dynamic programming approach introduced by Bellman, the LP approach suffers from poor scalability properties referred to as \textit{curse of dimensionality} \cite{BertsekasNDP}. The sources of intractability for systems with continuous state and action spaces can be identified as, among others, an optimization variable in an infinite dimensional space and an infinite number of constraints. For this reason the infinite dimensional LPs are usually approximated by tractable finite dimensional ones \cite{EsfahaniFromInftoFinitePrograms,deFariasLPapproach,SchweitzerPolynomApproxinMDP,PaulADP}.

In recent years, the LP approach has experienced an increasing interest, especially in combination with model-free control techniques \cite{SutterCDC2017,GoranADP,AndreaRelaxedOperator, AlexandrosIFAC20}. In such a setting, one assumes the dynamical system to be unknown but observable via state-space exploration, and builds one Bellman inequality (or constraint) of the LP for each observed transition. In this way, it is possible to both bypass the more classic system identification step and mitigate a source of intractability by solving an LP with a finite amount of constraints. A discussion on the approximation introduced by constraint sampling can be found in \cite{deFariasConstraintSampling}. Both theoretical and empirical evidence present in the previously mentioned literature suggest that a massive amount of data is generally needed to comply with a desired performance level.  As discussed in the scalability analysis performed in \cite{AndreaRelaxedOperator}, this aspect becomes even more evident for large-scale systems. Moreover, another relevant problem in the LP approach is to provide an estimate for the expected value in the constraints. This is usually performed by re-initializing the dynamics in the same state-input pairs and computing a Monte Carlo estimate of the associated value function evaluated at the next state \cite{SutterCDC2017}. In a stochastic framework, unfortunately, it may be practically impossible to re-initialize the system at desired states.

Another thriving data-driven research direction is the one revolving around behavioural theory and Willem's \textit{fundamental lemma} \cite{WillemsPersistence}, stating that the information contained in a sufficiently long trajectory of a linear system is, under mild assumptions, enough to describe any other trajectory of the same length that can be generated by the system itself. In a model-free context, a so-called \textit{persistently exciting} exploration input is often used to generate such trajectories, obtain a data-based representation of the underlying linear system and develop control techniques such as MPC \cite{JeremyDeePC} or assess system's properties such as dissipativity \cite{RomerDissipativity} and stabilizability/controllability \cite{TesiDatadrivencontrol}. The authors in \cite{vanWaardeDataInformativity} discuss the conditions for which a dataset is informative, \textit{i.e.} when the data contain enough information to accomplish a specific control task. 

Motivated by the poor scalability often affecting the LP approach and inspired by the recent literature on data-driven control of linear systems, in the present work we discuss how to mitigate the cost of performing massive exploration. After introducing the problem general formulation in Sec. \ref{Sec:optimalcontrol}, our main contributions can be summarised as follows:
\begin{itemize}
	\item We show in Sec. \ref{Sec:Unknowndynamics} that a sufficiently rich dataset can be used to generate all the constraints involved in the LP formulation for a linear system, offline and without observing the corresponding transitions;
	\item Moreover, thanks to a bilinear algebra framework, we show in Sec. \ref{Sec:Unknownstagecost} how to reconstruct the associated stage-cost evaluations starting again from a fixed dataset;
	\item For stochastic systems, we provide in Sec. \ref{Sec:Stochastic} insights on the estimation of the expected values in the constraints of the LP, without resorting to iterative dynamics re-initialization.
\end{itemize}

\subsection*{Notation and background}

We denote with $\mathbb{M}_p$ the set of $p \times p$ real matrices and with $\tilde{\mathbb{M}}_p \subset \mathbb{M}_p$ the subset of symmetric matrices. A vector of ones of suitable dimension is denoted with $\bm{1}$. 
The vectorization of a symmetric matrix $M \in \tilde{\mathbb{M}}_p$ with entries $[M]_{ij} = m_{ij}$ is $\mbox{vec}(M) = \begin{bmatrix} m_{11} & m_{12} & \cdots & m_{pp} 
\end{bmatrix} \in \mathbb{R}^{p^2}$ and its trace is denoted with $\mbox{tr}(M)$. A symmetric bilinear form \cite{BilinearAlgebra} is a map $\beta : \mathbb{R}^{p}\times\mathbb{R}^p \rightarrow \mathbb{R}$ that is linear in its arguments taken separately, and such that $\beta(x,y) = \beta(y,x)$ for all $x,y \in \mathbb{R}^p$.
%
%
A pair $(\mathbb{R}^p,\beta)$ defines a bilinear space. We also define the quadratic form $\ell : \mathbb{R}^{p} \rightarrow \mathbb{R}$ associated to $\beta$ as $\ell(z) = \beta(z,z)$ for all $z \in \mathbb{R}^p$. The following holds:
\begin{align}
\ell(a z) & = a^{2}\ell(z) \quad \forall a\in \mathbb{R}, \; \forall z\in \mathbb{R}^{p} \label{bilinear properties 1} \\
	\beta(x,y) & = \tfrac{1}{2}(\ell(x+y) - \ell(x) - \ell(y)) \quad \forall x,y \in \mathbb{R}^p. \label{bilinear properties 2}
\end{align}

Moreover, given a basis $\mathcal{B}=\{ b_1,\ldots,b_p \}$ for $\mathbb{R}^p$, we denote with $[M^{\mathcal{B}}]_{ij} = \beta(b_i,b_j)$ the matrix representation of $\beta$ in the basis $\mathcal{B}$.

\section{OPTIMAL CONTROL VIA LINEAR PROGRAMMING}\label{Sec:optimalcontrol}

Consider a discrete-time stochastic dynamical system
\begin{equation}\label{nonlinear dynamical system}
	x_{k+1} = f(x_{k},u_k,\xi_k),
\end{equation}
with (possibly infinite) state and action spaces $x_k \in \mathbb{R}^{n}$ and $u_k \in \mathbb{R}^{m}$. Here, $\xi_k \in \mathbb{R}^{n}$ denotes the realizations of independent identically distributed (i.i.d.) random variables, and $f : \mathbb{R}^n \times \mathbb{R}^m \times \mathbb{R}^n \rightarrow \mathbb{R}$ is the map encoding the dynamics.	
We consider \textit{stationary feedback policies}, given by functions $\pi : \mathbb{R}^n \rightarrow \mathbb{R}^m$; for more general classes of policies, see \cite{LasserreDTMCP}. A nonnegative cost is associated to each state-action pair through the \textit{stage cost} function $\ell : \mathbb{R}^n \times \mathbb{R}^m \rightarrow \mathbb{R}_{+}$. We introduce a \textit{discount factor} $\gamma \in (0,1)$ and consider the infinite-horizon cost associated to policy $\pi$
\begin{equation}\label{OC problem}
	v_{\pi}(x) = \E_{\xi}\left[  \sum_{k=0}^{\infty}\gamma^{k}\ell(x_k,\pi(x_k)) \; \bigg| \; x_0 = x \right].
\end{equation}
The objective of the optimal control problem is to find an optimal policy $\pi^{\ast}$ such that $v_{\pi^{*}}(x) = \inf_{\pi} v_{\pi}(x) = v^{*}(x)$, where $v^*$ is known as the optimal \textit{value function}. Let us define the vector space of all real-valued measurable functions that have a finite $r$-weighted sup-norm \cite[\S 2.1]{BertsekasAbstractDP} as
\begin{equation}\label{finite sup-norm}
 \mathbb{V} = \{ v : \mathbb{R}^n \rightarrow \mathbb{R} \;\; \vert \;\; || v ||_{\infty,r} <\infty \}.
\end{equation}
Throughout the paper, we work under \cite[Assump. 4.2.1 and 4.2.2]{LasserreDTMCP} to ensure that $v^{*} \in \mathbb{V}$, $\pi^{*}$ is measurable and the infimum of $v_{\pi}$ is attained. 
The optimal value function can be expressed as the solution of the following infinite-dimensional linear program \cite{LasserreDTMCP,deFariasLPapproach}
%
%
\begin{equation}\label{LPvaluefunction}
	\begin{aligned}
		\sup_{v\in \mathbb{V}} & \int_{\mathbb{R}^n} v(x)c(dx) \\
		\mbox{s.t.} \;\; & v(x) \le \ell(x,u) + \gamma \E_{\xi}v(f(x,u,\xi)) \quad \forall (x,u),
	\end{aligned}
\end{equation}
where $c$ is a finite measure that assigns positive mass to all open subsets of $\mathbb{R}^n$. The above formulation is not solvable in general due to several sources of intractability, see \textit{e.g.} \cite{PaulADP} and \cite{WangIteratedBellInequalities}. If one is nonetheless able to obtain $v^*$, they can in principle compute the corresponding policy by 
\begin{equation}\label{greedy policy}
	\pi^*(x) = \arg \min_{u} \{ \ell(x,u) + \gamma \E_{\xi}v^*(f(x,u,\xi)) \}.
\end{equation}
A special case of the infinite-horizon optimal control problem arises when the dynamics is linear 
\begin{equation}\label{linearmap}
	f(x,u,\xi) = Ax + Bu + \xi,
\end{equation}
with $A \in \mathbb{M}_n, B \in \mathbb{R}^{n \times m}$, and the cost function is quadratic
\begin{equation}\label{quadraticcost}
	\ell(x,u) = \begin{bmatrix} x \\ u \end{bmatrix}^{\intercal} L \begin{bmatrix} x \\ u \end{bmatrix} = \begin{bmatrix} x \\ u \end{bmatrix}^{\intercal} \begin{bmatrix}
		L_{xx} & L_{xu} \\ L_{xu}^{\intercal} & L_{uu}
	\end{bmatrix} \begin{bmatrix} x \\ u \end{bmatrix}. 
\end{equation}
For such linear-quadratic (LQ) problems we impose the following assumption. 
\begin{assumption}\label{assumption stabilizability}
	The pair $(A,B)$ is stabilizable, $\xi$ is i.i.d. with zero mean and covariance matrix $\Sigma$. Moreover, $L\succeq 0$ and $L_{uu} \succ 0$.
\end{assumption}
By denoting $P\in \tilde{\mathbb{M}}_n$ and $e\in\mathbb{R}^n$, let us define
\begin{equation}\label{function space v}
	\mathbb{V}_{q} = \{ v : \mathbb{R}^{n} \rightarrow \mathbb{R} \;\; \vert \;\; v(x) = x^{\intercal} P x + e \} \subset \mathbb{V}.                                                            
\end{equation} 
The solution to the LP \eqref{LPvaluefunction} under LQ assumptions \eqref{linearmap}-\eqref{quadraticcost} is then $v^*(x) = x^\intercal P^*x+e^* \in \mathbb{V}_{q}$, where $P^*$ is the solution to the well-known associated algebraic Riccati equation (ARE) \cite{DavisStochasticControl} and $e^*= \tfrac{\gamma}{1-\gamma}\mbox{tr}(P^*\Sigma)$. By imposing $c$ to be a probability distribution with zero mean and identity covariance matrix and restricting $v\in\mathbb{V}_q$, an equivalent formulation for \eqref{LPvaluefunction} that directly involves $P$ and $e$ is \cite{PaulADP}
\begin{equation}\label{LP1}
	\begin{aligned}
		\max_{P, \,e} & \quad \mbox{tr}(P) + e \\
		\mbox{s.t.} & \quad x^{\intercal}Px \le \ell(x,u) \\
		& \quad + \gamma \E_{\xi} (Ax+Bu+\xi)^{\intercal}P(Ax+Bu+\xi),
	\end{aligned}
\end{equation}
for all $(x,u)$. 

In model-based control, typically one assumes that $f$ is known and directly looks for solutions to, \textit{e.g.}, the ARE or finite-dimensional semidefinite programs using convex optimisation tools. In model-free control on the other hand, one assumes that the model is unknown but can be observed indirectly through data $(x^i,u^i,x^{i+})$ of state-input pairs $(x^i,u^i) \in \mathbb{R}^n \times \mathbb{R}^{m}$ and the next state $x^{i+}=f(x^i,u^i,\xi^i)$. The solution to \eqref{OC problem} can be estimated, \textit{e.g.}, by means of reinforcement learning methods \cite{SuttonRLanIntroduction}. Moreover, one can also obtain the optimal policy by reformulating \eqref{LPvaluefunction}-\eqref{greedy policy} in terms of the so-called $Q$-function \cite{WatkinsQLearning}. To keep the discussion simple, however, we do not explore this direction here. In the context of the LP approach, the infinite constraints needed to construct the feasible region in \eqref{LPvaluefunction} or \eqref{LP1} are often replaced with a finite subset, each one associated with a data tuple $(x^i,u^i,x^{i+})$, as argued in \cite{GoranADP,deFariasConstraintSampling,AndreaRelaxedOperator,AlexandrosIFAC20}.

\section{FEASIBLE REGION SYNTHESIS FROM DATA}

\subsection{Unknown dynamics}\label{Sec:Unknowndynamics}

The information contained in a sufficiently long trajectory of a linear system is, under mild assumptions, enough to describe any other trajectory of the same length that can be generated by the system \cite{WillemsPersistence}. In this section we discuss the consequences of this idea and adapt the theoretical implications to the context of the data-driven LP approach. Throughout the section we work under LQ assumptions \eqref{linearmap}-\eqref{quadraticcost}, Assumption \ref{assumption stabilizability}, and we consider deterministic dynamics, \textit{i.e.} $\xi=0$ for all time steps and $\Sigma = 0$. First, we provide a description in matrix form of the Bellman inequalities.

\begin{proposition}\label{prop}
	The constraint set in \eqref{LP1} is equivalent to 
	\begin{equation}\label{constraint set}
		\mbox{vec}(H(x,u))^{\intercal}\mbox{vec}(P) \le \ell(x,u) \quad \forall (x,u),
	\end{equation}
where $H : \mathbb{R}^{n}\times\mathbb{R}^{m} \rightarrow \tilde{\mathbb{M}}_n$ is
\begin{equation}\label{H(x,u)}
	H(x,u) = xx^\intercal - \gamma (Ax+Bu)(Ax+Bu)^\intercal.
\end{equation}
\end{proposition}
\begin{proof}
 We can express the constraint set in \eqref{LP1} as
	\begin{equation}\label{inequalityLQR}
		\sum_{i=1}^{n}\sum_{j=i}^{n}h_{ij}(x,u)p_{ij} \le \ell(x,u) \quad \forall (x,u),
	\end{equation}
	where 
	\begin{equation*}\label{hij LQR}
		h_{ij}(x,u) = x_ix_j - \gamma (Ax+Bu)_{i}(Ax+Bu)_{j}.
	\end{equation*}
Then by imposing $[H(x,u)]_{ij} = h_{ij}(x,u)$ we obtain \eqref{H(x,u)} and, finally, we can re-arrange the left-hand side of \eqref{inequalityLQR} in vector form and obtain \eqref{constraint set}.
\end{proof}
We say that $(X,U,X^{+})$ is a dataset of length $T$ when $X = \begin{bmatrix} x^{1} & \cdots & x^{T} \end{bmatrix}$, $U = \begin{bmatrix} u^{1} & \cdots & u^{T} \end{bmatrix}$ and $X^{+} = AX+BU$. We also introduce the following assumption.

\begin{assumption}\label{assumption rank}
	$\mbox{rank}\begin{bmatrix} X \\ U \end{bmatrix} = n+m$.
\end{assumption}

The following lemma shows how all the infinite constraints in \eqref{LP1} can potentially be reconstructed offline directly from the dataset without explicitly determining the matrices $A$ and $B$ and without observing the corresponding system's transitions.
\begin{lemma}\label{lemma}
	Consider a dataset $(X,U,X^{+})$ of length $T$ satisfying Assumption \ref{assumption rank}. Then, for each $(x,u) \in \mathbb{R}^{n}\times\mathbb{R}^{m}$ there exists an $\alpha \in \mathbb{R}^{T}$ 
and such that
	\begin{equation}\label{fundamental relation}
		H(x,u) = (X\alpha)(X\alpha)^\intercal - \gamma (X^{+}\alpha)(X^{+}\alpha)^\intercal.
	\end{equation}
\end{lemma}

\vspace{0.1cm}

\begin{proof}
	Since $\big[X^\intercal \;\, U^\intercal\big]^\intercal$ is full row-rank by assumption, we know there exists an $\alpha$ satisfying
	\begin{equation}\label{linear comb}
		\begin{bmatrix}
			x \\ u
		\end{bmatrix} = \begin{bmatrix} X \\ U
		\end{bmatrix} \alpha.
	\end{equation}
Moreover, since
\begin{equation}
	 AX\alpha + BU\alpha = (AX+BU)\alpha = X^{+}\alpha,
\end{equation}
we have that
\begin{align*}
		H(x,u) & = xx^\intercal - \gamma (Ax+Bu)(Ax+Bu)^\intercal \\ & = (X\alpha)(X\alpha)^\intercal - \gamma (X^{+}\alpha)(X^{+}\alpha)^\intercal,
\end{align*}
that concludes the proof.
\end{proof}
Thanks to Proposition~\ref{prop} and Lemma~\ref{lemma} we know that, given a dataset $(X,U,X^+)$ satisfying Assumption \ref{assumption rank}, we can reconstruct each of the infinite constraints in \eqref{LP1} by computing $H(x,u)$ at suitable linear combinations of $X$ and $X^{+}$. Figures~\ref{fig1} and~\ref{fig2} illustrate the fundamental role of the sampled constraints for the suboptimality of the derived solution and, consequently, the utility of the proposed artificial sampling technique on the linear system 
\begin{equation}\label{system numerics}
	x^+ = \begin{bmatrix}
		1 & 0.1 \\ 0.5 & -0.5
	\end{bmatrix}x + \begin{bmatrix}
		1 \\ 0.5
	\end{bmatrix}u.
\end{equation}
In particular, only the first 10 constraints are generated via simulations of the system. The collected data are then used to synthesise new constraints according to \eqref{fundamental relation} for randomly selected values of $\alpha$. As depicted in Fig.~\ref{fig1}, the additional constraints allow for a dramatic improvement of the optimality gap. A graphical representation of the support constraints displacement is plotted in Fig~\ref{subfig1} in the variables space $(p_{11},p_{12},p_{22})$. In particular, we first solve the LP with the 10 observed constraints (blue dot), and then we solve the LP again by including 10 additional constraints generated artificially (red dot). 
Fig.~\ref{subfig2} shows the corresponding improvement in the value function. As observed experimentally, in general generating enough constraints from exploration to reach a prescribed performance level can be prohibitive in a real scenario. Our proposed approach alleviates this issue by allowing one to only sample a small subset of the state-space and then inexpensively synthesise new constraints offline. 

Artificial constraints generation can also be exploited in a policy iteration (PI) fashion \cite{BertsekasVol2} to, \textit{e.g.}, complement the approach proposed in \cite{GoranADP}. The initialization of the data-driven PI algorithm can be performed by exploring the state-space, as described in \cite{GoranADP}. Then, at any successive step $t$ of the algorithm one can emulate the PI behaviour by selecting appropriate vectors $\alpha$ that target state-input pairs associated with the current policy as follows
\begin{equation}
	\begin{bmatrix}
		x \\ K_tx
	\end{bmatrix} = \begin{bmatrix} X \\ U
	\end{bmatrix} \alpha,
\end{equation}
without need for further exploration.

\begin{figure}
	\centering
	\includegraphics[width=0.9\linewidth]{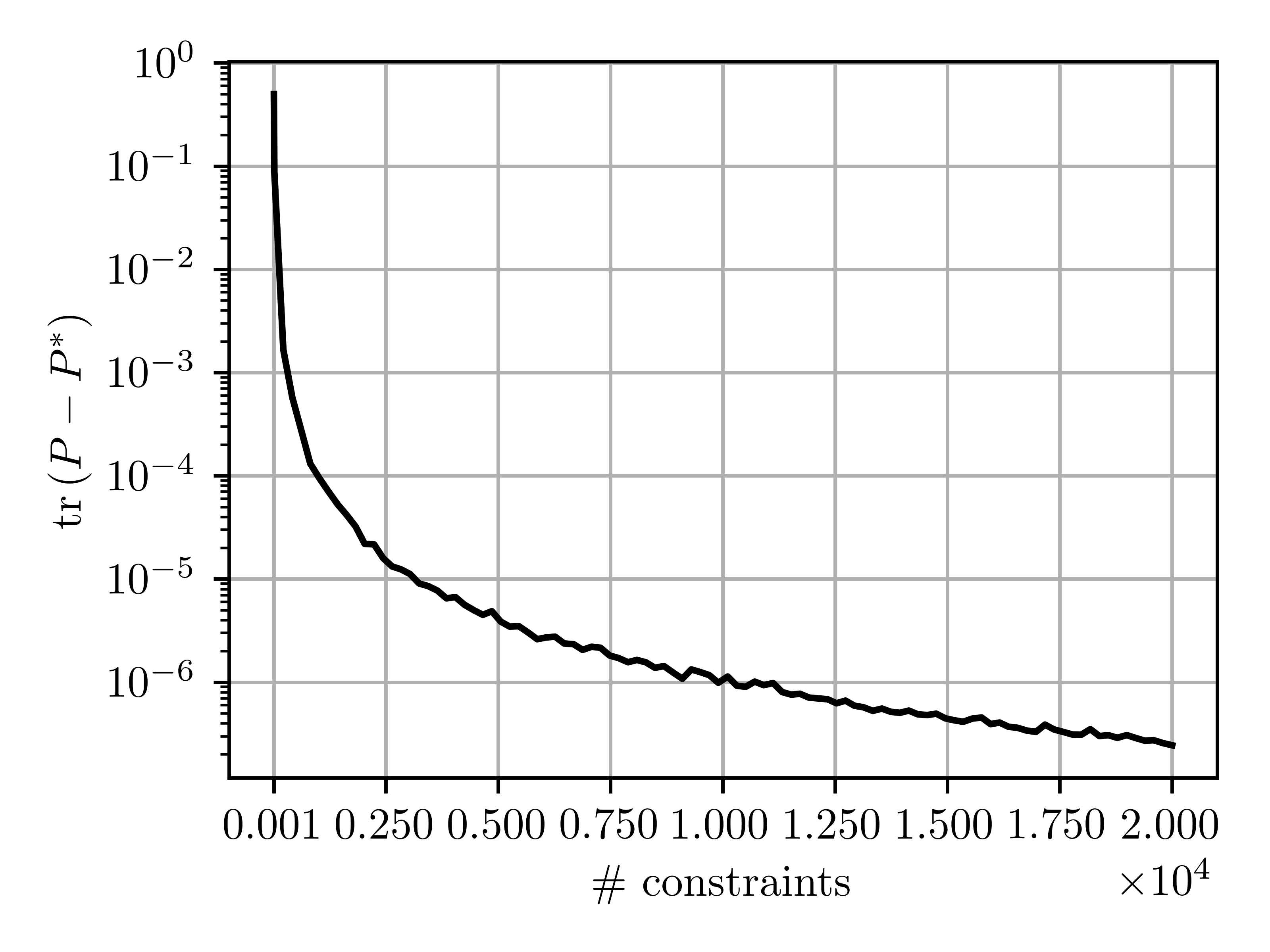}
	\caption{Median across $10^3$ independent runs of the optimality gap versus the number of constraints for system~\eqref{system numerics}.}
	\label{fig1}
\end{figure}

\subsection{Unknown stage cost}\label{Sec:Unknownstagecost}

In the context of control applications, the stage-cost is formulated by the designer: it is thus reasonable to consider $\ell(x,u)$ to be known. In cases where the stage-cost is not known, the following proposition provides a way to reconstruct $\ell(x,u)$ if we observe the cost incurred at a finite number of state-input pairs.

\begin{proposition}\label{prop:stagecost}
	Consider a dataset $(X,U,X^+)$ satisfying Assumption \ref{assumption rank}, and the square matrix $\big[\tilde{X}^\intercal \;\, \tilde{U}^\intercal \big]^\intercal \in \mathbb{M}_{n+m}$ obtained by down-selecting $n+m$ columns from $\big[X^\intercal \;\, U^\intercal\big]^\intercal$ such that Assumption~\ref{assumption rank} is still satisfied. Then, for each $(x,u) \in \mathbb{R}^{n}\times\mathbb{R}^{m}$ there exists an $\alpha \in \mathbb{R}^{n+m}$ such that
	\begin{equation}\label{stage cost decomposition}
		\ell(x,u) = \alpha^{\intercal} L_{X,U} \alpha,
	\end{equation}
where $L_{X,U} \in \tilde{\mathbb{M}}_{n+m}$ is
\begin{equation}\label{LXU}
	[L_{X,U}]_{ij} = \beta \Big(\begin{bmatrix} x^{i} \\ u^{i}
	\end{bmatrix}, \begin{bmatrix} x^{j} \\ u^{j}
	\end{bmatrix} \Big),
\end{equation}
and $\beta$ is the bilinear form associated to $\ell$.
\end{proposition}
\begin{proof}
	First note that there exists a unique $\alpha$ satisfying 
	\begin{equation}\label{linear comb quadratics}
		\begin{bmatrix}
			x \\ u
		\end{bmatrix} = \begin{bmatrix} \tilde{X} \\ \tilde{U}
		\end{bmatrix} \alpha.
	\end{equation}
Let us temporarily denote $z = \big[ x \;\, u\big]^\intercal$ and $Z = \big[\tilde{X}^\intercal \;\, \tilde{U}^\intercal \big]^\intercal$ such that $z^i$ is the $i$-th column of $Z$ and $\alpha^i$ is the $i$-th entry of $\alpha$. Moreover, we recall that, since $\ell : \mathbb{R}^{n+m} \rightarrow \mathbb{R}$ is a quadratic form, properties \eqref{bilinear properties 1}-\eqref{bilinear properties 2} hold. 
This allows us to express $\ell(x,u) = \ell(z) = \ell(Z\alpha)$ as
\small
\begin{align*}\label{bilinear form first}
	\ell(Z\alpha) & = \ell \Big(\sum_{i=1}^{n+m}\alpha^iz^i\Big) \\
	& = \ell(\alpha^1z^1) + \ell \Big( \sum_{i=2}^{n+m}\alpha^iz^i \Big) + 2\beta \Big(\alpha^{1}z^{1},\sum_{i=2}^{n+m}\alpha^iz^i\Big).
\end{align*}
\normalsize
On the other hand, it also holds
\small
\begin{equation*}
	\ell\Big(\sum_{i=2}^{n+m}\alpha^iz^i\Big) = \ell(\alpha^2z^2) + \ell\Big( \sum_{i=3}^{n+m}\alpha^iz^i \Big) + 2\beta\Big(\alpha^{2}z^{2},\sum_{i=3}^{n+m}\alpha^iz^i\Big).
\end{equation*}
\normalsize
Hence, $\ell(Z\alpha)$ can be written recursively as
\small
\begin{align}\label{recursive f}
		\ell(Z\alpha) & = \sum_{i=1}^{n+m}\ell(\alpha^iz^i) + 2\sum_{i=1}^{n+m} \beta\Big(\alpha^{i}z^{i},\sum_{j=i+1}^{n+m}a^jz^j\Big) \notag \\
		& = \sum_{i=1}^{n+m}{\alpha^i}^{2}\ell(z^i) + 2\sum_{i=1}^{n+m}\sum_{j=i+1}^{n+m}\alpha^i\alpha^j\beta(z^i,z^j) \notag \\
		& = \alpha^{\intercal} \begin{bmatrix} \ell(z^1) & \beta(z^1,z^2) & \cdots & \beta(z^1,z^k) \notag \\
			\star & \ell(z^2) & \cdots & \beta(z^2,z^k) \\
			\star&\star & \ddots & \vdots \\
			\star& \star& \star& \ell(z^k)
		\end{bmatrix} \alpha \\
	& = \alpha^{\intercal} L_{X,U} \alpha,
\end{align}
\normalsize
where the symbol $\star$ denotes symmetry. Finally, since $\ell(z^i)=\beta(z^i,z^i)$, we obtain \eqref{stage cost decomposition}-\eqref{LXU}.
\end{proof}

\begin{figure}
	\centering
	\begin{subfigure}[b]{0.45\textwidth}
		\centering
		\includegraphics[width=0.9\linewidth]{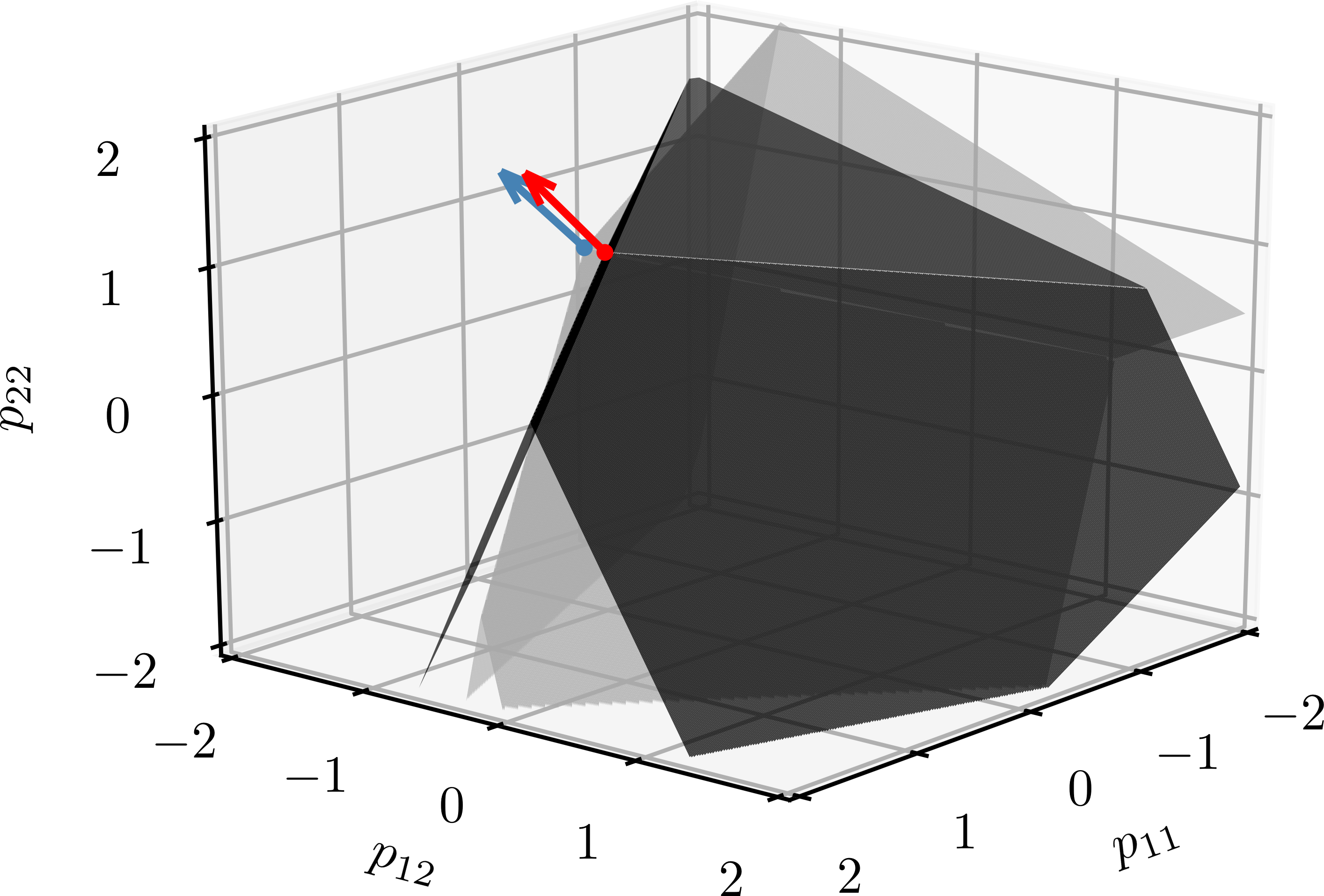}
		\caption{Relative displacement of the support constraints generated with exploration (light grey) and artificial sampling (dark grey) for the LP associated with system~\eqref{system numerics}. The colored dots and arrows represent the optimal solutions of the LPs and the gradient direction, respectively.}
		\label{subfig1}
	\end{subfigure}
	\hfill
	\begin{subfigure}[b]{0.45\textwidth}
		\centering
		\includegraphics[width=0.8\linewidth]{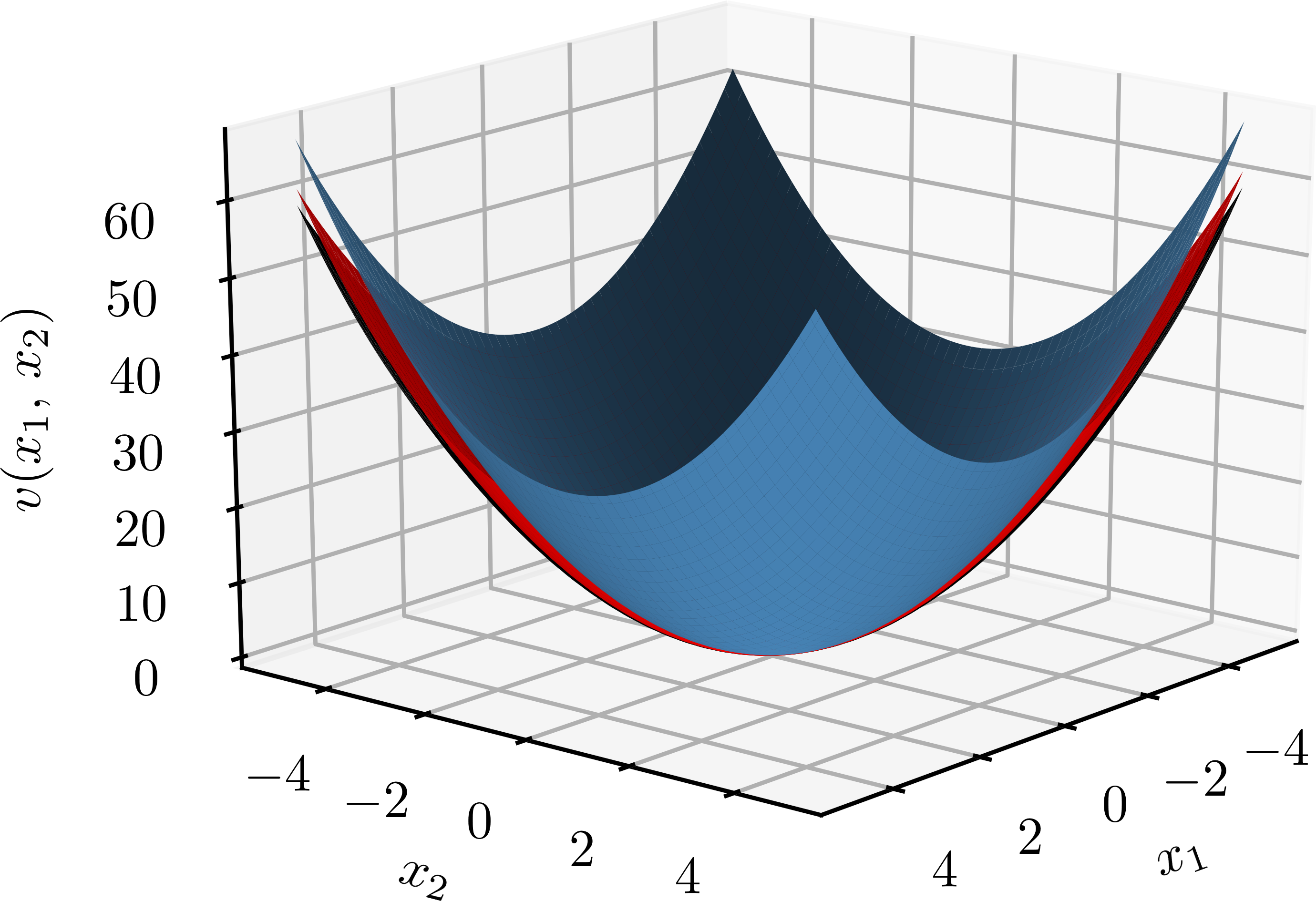}
		\caption{Color-matched graphical representation of the quadratics associated with the solutions depicted in Fig.~\ref{subfig1} and the optimal value function (in black).}
		\label{subfig2}
	\end{subfigure}
	\caption{Graphical comparison of the support constraints and solutions for the LPs associated with system~\eqref{system numerics} when different sets of constraints are considered.}
	\label{fig2}
\end{figure}

Note that if we express an arbitrary state-input pair $(x,u)$ as a linear combination of our data, as in \eqref{linear comb quadratics}, we can write
\begin{equation*}
	\ell(x,u) = \begin{bmatrix}
		x \\ u
	\end{bmatrix}^{\intercal} L \begin{bmatrix}
	x \\ u
\end{bmatrix} = \alpha^\intercal\underbrace{\begin{bmatrix}
\tilde{X} \\ \tilde{U}
\end{bmatrix}^{\intercal} L \begin{bmatrix}
\tilde{X} \\ \tilde{U}
\end{bmatrix}}_{L_{X,U}}\alpha.
\end{equation*}
It becomes evident that $L$ and $L_{X,U}$ are congruent matrices and therefore they are two matrix representations of the same quadratic form $\ell$ expressed in two different bases \cite{BilinearAlgebra}. Under this light, the data matrix $\big[\tilde{X}^\intercal \;\, \tilde{U}^\intercal \big]^\intercal$ takes the role of the matrix transforming the basis of $L$ into the basis of $L_{X,U}$.

Regarding the computation of $L_{X,U}$, according to \eqref{recursive f} and for each entry $(i,j)$ of $L_{X,U}$, we have to evaluate $\beta$ at the corresponding $(x^{i},u^{i})$, $(x^{j},u^{j})$ picked from our dataset. We already have the $n+m$ diagonal terms of $L_{X,U}$ as they are direct stage cost evaluations $\ell(x^i,u^i)$. As for the off-diagonal terms, by recalling once again Equation \eqref{bilinear properties 2}, we still have to add $\binom{n+m}{2}$ observations in our dataset, one for each pairwise combination of $z^i$ and $z^j$. The total amount of observations needed to compute $L_{X,U}$ is $n+m+\tbinom{n+m}{2} = \tfrac{(n+m)(n+m+1)}{2}$ that, expectedly, equals the amount of unknown entries in $L$.

\subsection{State-space exploration}

A dataset $(X,U,X^{+})$ satisfying Assumption \ref{assumption rank} can be generated by initializing the dynamics at desired states and applying the suitable inputs to ensure the rank condition is satisfied. In case targeted initialization is not possible, one can build independent samples by initializing the dynamics at an arbitrary state and running a long and rich enough exploration sequence, often guaranteed by the \textit{persistence of excitation} condition on the input \cite{WillemsPersistence,TesiDatadrivencontrol}. In detail, a sequence $u^1, \dots, u^T\in\mathbb{R}^{m}$ is said to be persistently exciting of order $L$ if the associated Hankel matrix of depth $L$,

\small
\begin{equation}
\mathcal{H}_{L} = \begin{bmatrix}
		u^1 & u^2 & \cdots & u^{T-L+1} \\
		u^2 & u^3 & \cdots & u^{T-L+2} \\
		\vdots & \vdots & & \vdots \\
		u^L & u^{L+1} & \cdots & u^T 
	\end{bmatrix} \in \mathbb{R}^{(mL) \times (T-L+1)},
\end{equation}
\normalsize
has full row rank $mL$. It is evident that such condition can only be satisfied if $T \ge L(m+1)-1$. Consider then to excite the system $x^+=Ax+Bu$ with a persistently exciting sequence $u^1, \dots, u^T \in \mathbb{R}^m$ of order $n+1$, implying that $T \ge n(m+1)+m$, and record the associated state transitions $x^1,\dots,x^{T}\in\mathbb{R}^n$. Then, \cite[Corollary 2.(ii)]{WillemsPersistence} ensures that
\begin{equation}\label{conditionrank}
	\mbox{rank} \begin{bmatrix}
		x^{1} & \cdots & x^{T} \\
		u^{1} & \cdots & u^{T}
	\end{bmatrix} = n+m,
\end{equation}
and Assumption \ref{assumption rank} is satisfied, as discussed in \cite{TesiDatadrivencontrol}.

Between targeted initialization and a single exploration sequence, we can mention the works in \cite{vanWaardeMultipleDatasets} and \cite{coulson2020distributionally} where condition \eqref{conditionrank} is guaranteed even when the dataset is composed by multiple (possibly short) roll-outs.

As mentioned in the introduction, in recent literature on data-driven control Willem's \textit{fundamental lemma} is often used to obtain a data-based representation of the trajectory space of a linear system and develop analysis and control techniques. Within the context of the LP approach, we show how to construct all infinite constraints compatible with system's dynamics starting from a sufficiently rich dataset, allowing one to avoid massive sampling. 

\section{FEASIBLE REGION ESTIMATION FOR STOCHASTIC SYSTEMS}\label{Sec:Stochastic}

In the context of the LP approach, a fundamental issue when dealing with stochastic systems is the estimation of the expected values in the Bellman inequalities. As discussed \textit{e.g.} in \cite{SutterCDC2017} and \cite{AndreaRelaxedOperator}, one could re-initialize the dynamics at a fixed state-input pair $(x,u)$ a sufficient number of times $N$, observe the corresponding transition $f(x,u,\xi)$ and estimate $\E_{\xi}[v(f(x,u,\xi))]$ by averaging the observations in a Monte Carlo fashion, as
\begin{equation}\label{approx monte carlo nonlinear}
	\frac{1}{N}\sum_{k=1}^{N}v(f(x,u,\xi^k)) \approx \E_{\xi}[v(f(x,u,\xi))].
\end{equation}
On the other hand, such an estimation can only be performed if one can re-initialize the dynamics at the same state $x$ and play the same input $u$ multiple times. Clearly this assumption is limiting in a stochastic framework, since it may be impossible to re-initialize the system at a desired state in general. Here we discuss the effect of removing the re-initialization assumption by averaging the observations over the next state $f(x,u,\xi)$ instead of over the value function $v(f(x,u,\xi))$ under LQ assumptions \eqref{linearmap}-\eqref{quadraticcost}.

First, we give a matrix description of the Bellman inequalities by specializing Proposition \ref{prop} to stochastic linear systems.

\begin{proposition}\label{prop stoch}
	The constraint set in \eqref{LP1} is equivalent to
	\begin{equation}\label{constraint set stoch}
		\mbox{vec}(\E_{\xi}G(x,u,\xi))^{\intercal}\mbox{vec}(P) \le \ell(x,u) \quad \forall (x,u),
	\end{equation}
where $G : \mathbb{R}^n\times \mathbb{R}^m \times \mathbb{R}^n \rightarrow \tilde{\mathbb{M}}_n$ is
\begin{equation}\label{G and H}
	G(x,u,\xi) = xx^\intercal - \gamma (Ax+Bu+\xi)(Ax+Bu+\xi)^\intercal.
\end{equation}
\end{proposition}
\begin{proof}
	The result holds by considering a similar reasoning to the one in the proof of Proposition \ref{prop}.
\end{proof}
As expected, note that in case of deterministic dynamics \eqref{constraint set stoch} reduces to \eqref{constraint set}, as $G(x,u,0) = H(x,u)$. For zero mean noise, the expectation of $G$ is given by
\begin{equation}\label{expected G zero mean}
	\E_{\xi}G(x,u,\xi) = H(x,u) - \gamma \Sigma,
\end{equation}
and the effect of the noise boils down to a constant term in the matrix of coefficients.

In the context of the data-driven LP approach, one could be tempted to directly use data observed from the evolution of the stochastic dynamics, constructing a set of noise-corrupted constraints of the form
	\begin{equation}\label{constraint set stoch wrong}
	\mbox{vec}(G(x,u,\xi))^{\intercal}\mbox{vec}(P) \le \ell(x,u).
\end{equation}
Alternatively, a Monte Carlo approach could be employed to mitigate the effect of the noise. In the linear context, the estimation with re-initialization \eqref{approx monte carlo nonlinear} corresponds to performing the following approximation
\begin{align}\label{reinitialization}
	\frac{1}{N}\sum_{k=1}^{N}G(x,u,\xi^k) \approx \E_{\xi}G(x,u,\xi).
\end{align} 
In this case, however, re-initialization can be circumvented by averaging the $x^+$ directly instead of $G$. Consider to have a dataset $(X,U,X^+)$ of sufficient length~$N$, where $\bar{x} = \tfrac{1}{N}X\bm{1}$ and $\bar{u} = \tfrac{1}{N}U\bm{1}$ are the average state and input while $\bar{x}^+=\tfrac{1}{N}X^+\bm{1}$ is the observed sample mean of the transition over the $N$ realizations included in the dataset.

\begin{proposition}\label{prop:montecarlo}
	Consider system \eqref{linearmap} and $v\in \mathbb{V}_{q}$. Then, 
	\begin{equation}\label{monte carlo on xplus}
		G(\bar{x},\bar{u},\bar{\xi}) = \bar{x}\bar{x}^\intercal - \gamma \bar{x}^+ {\bar{x}^{+\intercal}},
	\end{equation}
where $\bar{\xi} = \tfrac{1}{N}\sum_{k=1}^{N}\xi^k$ is the sample mean of the noise over the corresponding $N$ realizations.
\end{proposition}
	\begin{proof}
		Thanks to the linearity of the dynamics it holds
		\begin{align*}
			\bar{x}^+ = \frac{1}{N}\sum_{k=1}^{N}(Ax^{k}+Bu^{k}+\xi^{k}) = A\bar{x} + B\bar{u} + \bar{\xi},
		\end{align*}
		hence
		\begin{align*}
			G(\bar{x},\bar{u},\bar{\xi}) & = \bar{x}\bar{x}^\intercal - \gamma (A\bar{x}+B\bar{u}+\bar{\xi})(A\bar{x}+B\bar{u}+\bar{\xi})^\intercal \\
			& = \bar{x}\bar{x}^\intercal - \gamma \bar{x}^+\bar{x}^{+\intercal},
		\end{align*}
		concluding the proof.
	\end{proof}
Note that in order to compute $G(\bar{x},\bar{u},\bar{\xi})$ as in \eqref{monte carlo on xplus} we do not need to know $\bar{\xi}$ which is, as a matter of fact, unknown and embedded in the dynamics; we only need to compute $\bar{x}$ and $\bar{x}^+$ from our dataset. We also stress that \eqref{monte carlo on xplus} holds for any dataset irrespective of Assumption \ref{assumption rank}. Therefore, we can substitute approximation \eqref{reinitialization} with the following

\begin{equation}\label{label1}
	G(x,u,\bar{\xi}) \approx G(x,u,0) = H(x,u),
\end{equation}
where the upper bar notation has been removed to stress the fact that \eqref{monte carlo on xplus} can in principle be computed for arbitrary $(x,u)$ depending on the available data.

According to equation \eqref{expected G zero mean}, we are neglecting the contribution of the covariance matrix into the constraints, ending up solving a sampled version of the LP for the deterministic system $x^+=Ax+Bu$. It is well-known \cite{DavisStochasticControl} that the difference between the optimal value function for a stochastic linear system and its corresponding deterministic one is a constant shift depending on the covariance matrix. Consequently, the two associated policies (see Eq. \eqref{greedy policy}) coincide. 

Therefore, we circumvented the re-initialization condition at the expense of approximating a value function that, asymptotically in the number of sampled constraints, is the one associated with the deterministic dynamics. In case one is interested in policy search only, this heuristic could represent a viable choice since it asymptotically preserves the optimal policy. In general, on the other hand, we want to stress that computing the associated policy with \eqref{greedy policy} requires knowledge of $f$. For this reason, in order to make this heuristic operational, one should first reformulate the LP in terms of $Q$-functions, such that the policy extraction does not depend on the matrices $A$ and $B$. 

Finally, we discuss a heuristic on constructing approximated artificial constraints. Consider to have a dataset of length $NT$ and to partition the data into $T$ subsets $(X^i,U^i,X^{i+})$ of length $N$, so that $X^{i+}=AX^i+BU^i+D^i$ and $D^i=[\xi^{i1} \; \dots \; \xi^{iN}]^\intercal$ contains the corresponding noise realizations. No rank assumption is needed on any of the $T$ datasets. Then, compute the average dataset $(\bar{X},\bar{U},\bar{X}^+)$, where $\bar{X}=[\bar{x}^1 \; \dots \; \bar{x}^N ]^\intercal$, $\bar{U}=[\bar{u}^1 \; \dots \; \bar{u}^N ]^\intercal$, $\bar{X}^+=[\bar{x}^{+1} \; \dots \; \bar{x}^{+N} ]^\intercal$, and each of their columns is $\bar{x}^i=\tfrac{1}{N}X^i\bm{1}$, $\bar{u}^i=\tfrac{1}{N}U^i\bm{1}$ and $\bar{x}^{+i}=\tfrac{1}{N}X^{+i}\bm{1}$. Since the partition into datasets is arbitrary it is reasonable to consider that Assumption \ref{assumption rank} can be satisfied for $(\bar{X},\bar{U},\bar{X}^+)$. 
Note that, for $N$ sufficiently large, we can approximate $\bar{X}^+ \approx A\bar{X}+B\bar{U}$. As a consequence, we can exploit \eqref{fundamental relation} to artificially build a desired number of (approximated) constraints associated with the deterministic system $x^+ = Ax + Bu$. Further discussion is necessary to provide probabilistic performance bounds on the error introduced and it is deferred to future studies.

\section{CONCLUSIONS AND FUTURE WORK}

On the wave of the exciting recent literature on behavioural theory, we showed how to synthesise new constraints for the LP formulation of a linear system starting from a suitable dataset. In this way, the often poor scalability properties of the LP approach are partially alleviated by generating constraints offline and without observing the dynamics evolution. Other significant insights were given about reconstructing the associated unknown stage-costs.

Many important issues are still to be explored and discussed, such as extending the approach to the $Q$-function formulation and relaxing the linearity assumptions on the dynamics to affine or polynomial.

\bibliographystyle{plain}
\bibliography{../autosam}

\end{document}